\documentclass[a4paper,english,11pt]{article}
\usepackage{amscd}

\usepackage{amsmath}
\usepackage{amsfonts}
\usepackage{verbatim}
\usepackage{amsthm}
\usepackage{amssymb}
\usepackage{amscd}
\usepackage{latexsym}
\usepackage{array}
\usepackage{enumerate}
\usepackage{longtable}

\ifx\pdfpagewidth\undefined 
\usepackage[T1]{fontenc}
\else 
\usepackage[OT1]{fontenc} 
\fi 
\usepackage{amssymb,amsfonts,amstext}

\begin{document}

\newtheorem{Theorem}{Theorem} 
\newtheorem{Corollary}[Theorem]{Corollary}
\newtheorem{Application}[Theorem]{Application}
\newtheorem{Definition}[Theorem]{Definition}
\newtheorem{Remark}[Theorem]{Remark}
\newtheorem{Example}[Theorem]{Example}
\newtheorem{Notation}[Theorem]{Notation}
\newtheorem{Lemma}[Theorem]{Lemma}
\newtheorem{Conjecture}[Theorem]{Conjecture}
\newtheorem{Proposition}[Theorem]{Proposition}
\newtheorem{OQ}[Theorem]{Open Question}

\newcommand{\E}{{\rm E}}
\newcommand{\R}{\mathbb R}
\newcommand{\C}{\mathbb C}
\newcommand{\Z}{\mathbb Z}
\newcommand{\T}{\mathbb T}
\newcommand{\N}{\mathbb N}
\newcommand{\M}{\mathbb M}
\newcommand{\Q}{\mathbb Q}
\renewcommand{\P}{\mathbb P}

\newcommand{\lra}{\longrightarrow}

\renewcommand{\phi}{\varphi}
\renewcommand{\theta}{\vartheta}
\newcommand{\eps}{\varepsilon}
\newcommand{\wt}[1]{\widetilde{#1}}
\newcommand{\phit}{\wt{\phi}}
\newcommand{\Times}{{\bigtimes}}
\renewcommand{\r }{\rangle}
\renewcommand{\l }{\langle}
\newcommand{\e}{\equiv}
\newcommand{\A}{{\rm a}}
\newcommand{\Dim}{{\rm dim}\,}
\newcommand{\ol}[1]{\overline{#1}}
\newcommand{\dis}{\displaystyle}
\newcommand{\D}{{\rm d}}
\newcommand{\lcm}{{\rm lcm}}
\newcommand{\plim}{ \lim_{\leftarrow}}
\newcommand{\g}{>}
\newcommand{\s}{<}
\newcommand{\np}{\noindent{\bf Proof}:$\quad $}
\renewcommand{\deg}{{\rm deg}}

\newenvironment{romliste}{\renewcommand{\labelenumi}{\rm\roman{enumi})}
                          \begin{enumerate}}{\end{enumerate}}

\title{ One--parameter subgroups  of topological abelian groups }
\date{}
\author{M. J. Chasco\\{\small\em Dept. de F\'{\i}sica y Matem\'{a}tica Aplicada},
{\small\em University  of Navarra}\\
{\small\em e-mail: mjchasco@unav.es}}

\maketitle
\begin{abstract}
It is proved that, for a wide class of topological abelian groups (locally quasi--convex groups for which the canonical evaluation from the group into  its Pontryagin bidual group is onto) the arc component of the group is exactly the union of the one--parameter subgroups. We also prove that for metrizable separable locally arc--connected reflexive groups, the exponential map from the  Lie algebra into the group is open.

\end{abstract}

Some properties of one--parameter subgroups of locally compact groups have been known for a long time. In a paper published  1967 Rickert proved  that in a compact arc--connected group every point lies on a one--parameter subgroup.  Previously Gleason   had shown  in 1950 that every finite dimensional, locally compact group contains a one-parameter subgroup. There are also examples of topological groups without non trivial one parameter subgroups; this  is the case for instance of the subgroup of integer--valued functions of the Hilbert space $L^2[0,1]$. For topological abelian groups which are $k-spaces$, the arc component of the dual group is the union of its one--parameter subgroups. This result  published by  Nickolas in 1977, remained for some years the only available piece of information outside the class of locally compact groups. Recently  it was proved in \cite{AChD} that the same is true for a much wider class of topological groups.

In this paper we go deeper into the study of one parameter subgroups of topological abelian groups. We take as a reference the papers \cite{AChD}, \cite{ATChD}, \cite{Nickolas}) and  the book by Hofmann and Morris ``The structure of compact groups" (2000) which presents a wide range of   properties of  the one--parameter subgroups  of locally compact abelian groups.

 In order to do that we  use two ingredients:  on one hand Pontryagin duality techniques and on the other the relation between ${\rm CHom}(\mathbb R,G)$ and the group $G$ which is given by the evaluation mapping $${\rm CHom}(\mathbb R,G)\longrightarrow G,\
\phi\longmapsto \phi(1)$$
  The vector space ${\rm CHom}(\mathbb
   R,G)$  endowed with the compact open topology is called the Lie algebra of the topological group $G$ and denoted by ${\cal L}(G)$ in analogy with the classical theory of Lie groups. In that case the evaluation mapping is continuous and it is called  the exponential function (${\rm exp}_G$).
  The elements of ${\rm
im}\, {\rm exp}_G$ are those lying on one-parameter subgroups, and $G$ is the union of its one-parameter subgroups if and only if ${\rm exp}_G$ is onto.

For a topological group $G$, we denote by $G_a$ its arc--component.

The main results of the paper are:

{\bf Theorem 6.}
 Let $G$ be a Hausdorff locally quasi--convex topological abelian group for which the evaluation mapping from the group into its bidual group  is onto. Then,  ${\rm im}\,{\rm exp}_{G}=G_a$

 {\bf Theorem 11.}
If $G$ is a  metrizable separable reflexive topological abelian group which has an arc--connected zero neighborhood, then the exponential mapping $\exp_{G}:{\cal L}(G)\longrightarrow G_a$ is a quotient map.

Observe that if $G$ is the additive topological group of a topological vector space $E$, then $\exp_{G}:{\rm CHom}(\R,G)\longrightarrow G,\
\phi\longmapsto \phi(1)$ is a topological isomorphism of topological vector spaces  and in that case all the topological and algebraic information about $G$ is in  ${\rm CHom}(\R,G)$.

We will use  Pontryagin duality theory,   hence  our topological groups need to be  abelian.
Pontryagin-van Kampen duality theorem for locally compact abelian  groups (LCA groups) is a very deep result. It is the  base of Harmonic Analysis and it allows to know the structure of LCA groups. This explains why a  consistent Pontryagin-van Kampen duality theory has been developed for general topological abelian groups  and why the abelian topological groups satisfying the Pontryagin-van Kampen duality, the so called  {\em reflexive groups}, have received considerable attention from the late 40's of the last century (see \cite{survey} for a survey on the subject).

For an abelian topological group $G$, the {\em character group} or
 {\em dual group}
 $G^\wedge$ of $G$ is defined by

 $$G^\wedge:=\{\varphi:G\rightarrow\T:\  \chi\ \mbox{ is a continuous
 homomorphism}\}$$ where $\T$ denotes the compact group of complex numbers
 of modulus $1$.  The elements of $G^\wedge$
 are called {\em characters}.
 We say that the group $G$ {\em has enough characters} or that $G$ is a {\em MAP} group if for any $0\neq x\in
 G$, there is some character $\varphi\in G^\wedge$ such that
 $\varphi(x)\neq 1$.

 Endowed with the compact--open topology $G^\wedge$ is an abelian
 Hausdorff group. The {\em bidual} group  $G^{\wedge\wedge}$ is $(G^\wedge)^\wedge$ and the canonical mapping in the bidual group is defined by:
 $$\alpha_G:G\rightarrow G^{\wedge\wedge},\ x\mapsto
 \alpha_G(x):\varphi\mapsto \varphi(x) .$$

 The group $G$ is called {\em Pontryagin reflexive} if $\alpha_G$ is
 a topological isomorphism. For the sake of simplicity, we will use the term
{\em reflexive} only. The famous Pontryagin--van Kampen theorem
states that every locally compact abelian (LCA) group is reflexive.

 The {\em annihilator of} a subgroup $H \subset G$ is defined as the subgroup
$H^\bot := \{ \varphi \in G^\wedge \colon \varphi
(H)=\{1\} \}$. If $L$ is a subgroup of $G^\wedge$, the {\it inverse annihilator} is defined  by $L^\bot:= \{  g \in G  \colon \varphi (g) = 1, \forall \varphi \in L \}$.

Annihilators are the particularizations for subgroups of the more general notion of polars of subsets. Namely, for $A \subset G $ and $B\subset G^\wedge$, the polar of $A$ is  $A^\triangleright: = \{ \varphi \in G^\wedge: \varphi(A) \subset \Bbb T_+ \}$  and the inverse polar of $B$ is $B^{\triangleleft}:= \{  g \in G  \colon \varphi (g) \in \Bbb T_+, \forall \varphi \in B \}$, where $\T_+:=\{z\in\T:\ {\rm Re}\,z\ge
0\}$ .

For a topological
Abelian group $G$, it is not difficult  to prove that a set $M\subset G^\wedge$ is
\emph{equicontinuous\/} if there exists a neighborhood $U$ of the
neutral element in $G$ such that $M\subset U^\triangleright$. Other  standard facts in duality theory are:  If $U$ is a neighborhood  zero in $G$, its polar is compact and  the dual group $G^\wedge$ is the union of all polars of zero neighborhoods   in $G$. The family $\{K^\triangleright; K \;\mbox{is a compact subset of}\;G  \}$ is a neighborhood basis  of the neutral element in $G^\wedge$.

A subgroup $H$ of a topological group $G$ is said to be {\em dually closed}   if, for every element $x$ of $G\setminus H$, there
is a continuous character $\varphi$ in $G^\wedge$ such that
$\varphi (H)=\{1\}$
and $\varphi (x)\not= 1$.

  Reflexive groups lie in a wider class of groups, the so called {\it locally quasi-convex groups}.  Vilenkin  had the seminal idea to define a sort of convexity for abelian topological groups inspired on the Hahn-Banach theorem for locally convex spaces.

A subset $A$ of a topological group $G$ is called {\em quasi-convex} if for every $g\in G\setminus A$, there is some $\varphi \in A^\triangleright $   such that $\mbox{Re}\varphi(g)<0$.

It is easy to prove that for any subset $A$ of a topological group $G$,  $A^{\triangleright \triangleleft}  $ is a quasi-convex set. It will be called the {\it quasi-convex hull} of $A$ since it is the smallest quasi-convex set  that contains $A$. Obviously, $A$ is quasi-convex iff $A^{\triangleright \triangleleft} = A$.

If $A$ is a subgroup of $G$, $A$ is  quasi-convex if and only if $A$ is dually closed. The abelian topological group $G$ is said to be locally quasi-convex if it admits a neighborhood basis of zero formed by quasi-convex subsets. Dual groups are examples of locally quasi-convex groups. For locally quasi--convex groups the evaluation map $\alpha_G$ is injective and open onto its image. A topological vector space $E$ is locally convex if and only if in its additive structure is a locally quasi-convex topological group (see \cite[2.4]{Ban}).

By a {\em  real character} (as opposed to a {\em character}) on an
abelian topological group $G$ it is commonly understood a continuous
homomorphism from $G $ into the reals $\mathbb{R }$.  The real
characters on $G$ constitute  the vector space  ${\rm CHom}(G,\mathbb
R)$. It is said that the group $G$ has {\em enough real characters}
if ${\rm CHom}(G,\mathbb R)$ separates the points of $G$. We denote by ${\rm CHom}_{co}(G,\mathbb R)$
the group of real characters endowed with the compact open topology. We say that a
character $\varphi: G\rightarrow \mathbb{T}$ can be \emph{lifted to
a real character}, if there exist a real character $f$ such that $e^{2\pi if}=\varphi$.
 We denote
by $ G^{\wedge}_{\rm lift}$ the subgroup of $G^\wedge$ formed by the characters that can be lifted to a continuous real
character.

Given a topological abelian group $G$, we denote by
$\omega(G,G^\wedge)$ the weak topology on $G$  that is,  the
topology on $G$ induced by the elements of $G^\wedge$. This topology
coincides with the Bohr topology.

On the other hand
$\omega(G^\wedge, G)$ denotes the topology on $G^\wedge$ of
pointwise convergence.

In a topological  group $G$, the arc components are homeomorphic and
it makes sense to refer to the arc component of the neutral element
$G_a$  as the arc component.

There are well known results about the
arc component of locally compact groups.
Let $G$ be an LCA group,  $G_a$ its arc component and  $G_0$ its
connected component.

\begin{enumerate}
\item
$G_a= {\rm im}\,{\rm exp}_{G}$ (see \cite[8.30]{HM}).
\item
The arc component $G_a$ is dense in the connected component $G_0$ (see \cite[7.71]{HM}).

\item
The group $G$ has enough real characters iff the dual group $G^\wedge$ is connected (see \cite[24.35]{HR}).
\end{enumerate}

What can be said about the arc component in more general classes of groups?
It is known that for topological abelian groups which are $k-spaces$,
the arc component of the dual group is exactly  the union of its one--parameter subgroups  (see \cite{Nickolas}).
It was proved recently  that the same is true for a much wider class of topological groups: groups satisfying the EAP condition (A topological group $G$ satisfies the EAP  if every  arc in $G^\wedge$ is equicontinuous). This  is introduced  and studied for different groups in \cite{AChD}.

Next, we find new classes of groups for which $G_a= {\rm im}\,{\rm exp}_{G}$.

 We start with a Lemma that  can be proved in a straightforward way.
\begin{Lemma}\label{dense}
Let $G$ be a Hausdorff topological abelian group,  $H$ a
subgroup and $L$ a  subgroup of $G^\wedge$ then,
\par
$H$ is dense in the weak topology $\omega(G,G^\wedge)$ iff
$H^\bot=\{e\}$
\par
 $L$ is dense in the pointwise topology $\omega(G^\wedge, G)$ if $L^\bot=\{e\}$.
If $G^\wedge$ is a MAP group,  the
reverse implication is also true.
\end{Lemma}

\begin{Lemma}\label{arco}
Let $G$ be a Hausdorff topological abelian group. If  $K$ is a compact
subset of $G$, then $\alpha_G(K)$ is an equicontinuous subset of $G^{\wedge\wedge}$.
\end{Lemma}
\begin{proof}
The polar set $K^\triangleright$ is a neighborhood of $0$ for the compact open topology of $G^\wedge$, then $K^{\triangleright\triangleright}$ is an equicontinuous subset of $G^{\wedge\wedge}$. Moreover, $\alpha_G(K^{\triangleright\triangleleft})=K^{\triangleright\triangleright}\bigcap \alpha_G(G)$. Therefore
$\alpha_G(K)\subseteq \alpha_G(K^{\triangleright\triangleleft})=K^{\triangleright\triangleright}\bigcap \alpha_G(G)\subseteq K^{\triangleright\triangleright}$. Since $\alpha_G(K)$ is contained in an equicontinuous subset, it is itself equicontinuous.
\end{proof}


\begin{Proposition}\label{continuous}
Let $G$ be a Hausdorff topological abelian group and $\gamma: \mathbb I\rightarrow G$ a continuous arc, then the mapping $\Phi_\gamma: G^\wedge \times \mathbb I \rightarrow \mathbb T$ defined by $\Phi_\gamma(\chi, t )=\chi(\gamma(t))$ is continuous.
\end{Proposition}
\begin{proof}
Take $\chi_0 \in G^\wedge$ and $t_0\in \mathbb I$. Let us see that  $\Phi_\gamma$ is continuous at $(\chi_0,t_0)$.
For every $n\in \N$, let us denote by $\T_n$ the neighborhood of $1$ in $\T$, $ \{e^{2\pi i t}:\   | t|\le  \frac{1}{4n}\}.$
The mapping $\chi_0\circ \gamma:\mathbb I \rightarrow \mathbb T$ is continuous so, for every $n\in \N$ there exists $V_{t_0}$ neighborhood of $t_0$ in $\mathbb I$ such that $\chi_0(\gamma(t))\overline{\chi_0(\gamma(t_0))} \in \mathbb T_{{2n}}$,  for every $t\in V_{t_0}$.\\
On the other hand since $\gamma(\mathbb I)$ is a compact subset of $G$, by \ref{arco}, $\alpha_G(\gamma(\mathbb I))$ is equicontinuous at $\chi_0$ hence, for every $n\in \N$, there exists a neighborhood $U_{\chi_0}$ of $\chi_0$ in $G^\wedge$ such that $\chi(\gamma(t))\overline{\chi_0(\gamma(t))} \in \mathbb T_{{2n}}$ for every  $t\in \mathbb I$ and $\chi \in U_{\chi_0}$.\\
Therefore $\chi(\gamma(t))\overline{\chi_0(\gamma(t))}\chi_0(\gamma(t))\overline{\chi_0(\gamma(t_0))} \in \mathbb T_{{2n}}\mathbb T_{{2n}}\subset \mathbb T_{{n}}$ for every $t\in V_{t_0}$  and $\chi \in U_{\chi_0}$. This proves  $\Phi_\gamma$ is continuous at $(\chi_0, t_0)$.
\end{proof}

\begin{Proposition}\label{contained}
Let $G$ be a Hausdorff topological abelian group and $G_a$ its arc component, then $G_a\leq \alpha_G^{-1}(G^{^{\wedge\wedge}}_{\rm lift})$

\end{Proposition}
\begin{proof}
Let $x$ be an element in $G_a$,  and let  $\gamma:\mathbb I\to G$ be a continuous mapping joining $0$ and $x$. Then,  $\Phi_\gamma: G^\wedge \times \mathbb I \rightarrow \mathbb T$ defined by $\Phi_\gamma(\chi,t )=\chi(\gamma(t))$ is continuous as we have seen in  \ref{continuous}. Denote by $\psi: G^\wedge \times\{0\} \to \R$ the null real character. By the homotopy lifting property we can find an homotopy $F:G^\wedge\times \mathbb I\to \R$ such that $p\circ F = \Phi_\gamma$ and $F\mid_{G^\wedge\times\{0\}}= \psi$. Now the unique path lifting property of $p:\R \to \T$ allows us to show that  $\varphi: G^\wedge \to \R$ defined as the restriction of $F$ to $G^\wedge\times \{1\}$   is a homomorphism and hence a continuos real character lifting $\alpha_G(x)$: $p\circ \varphi(l)=p\circ F(l,1)=\Phi_\gamma(l,1)=l(\gamma(1))=l(x)=\alpha_G(x)(l)$, for all $l\in G^\wedge$.
\end{proof}

\begin{Proposition}\label{alfasobre}
Let $G$ be a Hausdorff locally quasi--convex topological abelian group for which $\alpha_G$ is onto. Then,   $\alpha_G^{-1}(G^{^{\wedge\wedge}}_{\rm lift})\leq {\rm im}\,{\rm exp}_{G}\leq G_a$.
\end{Proposition}
\begin{proof}
Observe that ${\rm im}\,{\rm exp}_{G}$ is arc-connected and so it is contained in $G_a$. Let $x\in G$ such that $\alpha_G(x)\in G^{^{\wedge\wedge}}_{\rm lift}$ and a continuous homomorphism  $\widetilde{\alpha_G(x)}: G^\wedge \rightarrow \mathbb R$ such that $p\circ \widetilde{\alpha_G(x)}=\alpha_G(x)$.
Denote by $S$ the topological isomorphisms $S:\R\to \R^\wedge,\ s\to \chi_s:\chi_s(t)=e^{2\pi i s t}$.
Since $G$ is locally quasi--convex, the evaluation mapping $\alpha_G$ is  injective and open, so we can consider the homomorphism $f\in{\rm CHom}(G, \mathbb R)$ given by $f=\alpha_G^{-1}\circ (\widetilde{\alpha_G(x)})^\wedge\circ S$. Let us see that $f(1)=x$.
Observe first that $(\widetilde{\alpha_G(x)})^\wedge(S(1))=\alpha_G(x)$ [for $\psi\in G^\wedge$: $(\widetilde{\alpha_G(x)})^\wedge(S(1))(\psi)=(\chi_1\circ \widetilde{\alpha_G(x)})(\psi)=e^{2\pi i \widetilde{\alpha_G(x)}(\psi)}=(p\circ \widetilde{\alpha_G(x)})(\psi)=\alpha_G(x)(\psi)$]. Therefore, $f(1)=\alpha_G^{-1}\circ (\widetilde{\alpha_G(x)})^\wedge (S(1))=\alpha_G^{-1}(\alpha_G(x))=x$.
\end{proof}
The following Theorem, obtained from the previous propositions, shows in particular that in an arc--connected locally quasi--convex topological abelian group for which $\alpha_G$ is onto, every point lies in one one--parameter subgroup. The same was proved by Rickert in 1967 for compact arc--connected groups (see \cite{Ric}).
\begin{Theorem}\label{main}
Let $G$ be a Hausdorff locally quasi--convex topological abelian group for which $\alpha_G$ is onto. Then,  ${\rm im}\,{\rm exp}_{G}=G_a=\alpha_G^{-1}(G^{^{\wedge\wedge}}_{\rm lift})$
\end{Theorem}
Some classes of Hausdorff locally quasi--convex topological abelian group for which $\alpha_G$ is onto are the following:
reflexive groups, duals of pseudocompact groups, $P$-groups (see \cite{survey}), groups of continuous functions $C(X,\T)$ where $X$ is a completely regular $k$-space (see \cite[14.8]{Diss}) and the wide class of nuclear complete topological abelian groups \cite[21.5]{Diss}.

The class of nuclear groups was formally introduced by Banaszczyk in \cite{Ban}. It is   a class of topological groups   embracing nuclear spaces and locally compact abelian groups (as natural generalizations of finite-dimensional vector spaces). The definition  of nuclear groups
 is  very   technical, as could be expected from its virtue of  joining together objects of such different classes. A nice survey on nuclear groups is also provided by      L. Au$\ss$enhofer in \cite{Lydia2}.   The following are important facts concerning the class of nuclear groups:\\

 1. Nuclear groups are locally quasi-convex, $\cite[8.5]{Ban}$.

 2. Products, subgroups and quotients of nuclear groups are again nuclear, $\cite[7.5 ]{Ban}$.

3. Every locally compact abelian group is nuclear, $\cite[7.10]{Ban}$ .

4. Every closed subgroup in a nuclear topological group is dually closed $\cite[8.6 ]{Ban}$.

5. A nuclear locally convex space is a nuclear group, $\cite[7.4]{Ban}$. Furthermore, if a topological vector space $E$ is a nuclear group, then it is a locally convex nuclear space, $\cite[8.9]{Ban}$.

\begin{Proposition}\label{LieLydia}\cite[1.4]{LydiaLie}
Let $G$ be a Hausdorff locally quasi--convex topological abelian group for which $\alpha_G$ is continuous, then   the mapping $\Phi_0:{\cal L}(G)\to {\rm CHom}_{co}(G^\wedge,\mathbb{R^\wedge})$ given by $\varphi\mapsto \varphi^{\wedge}$: $\varphi^{\wedge}(\chi)=\chi\circ \varphi$ is an embedding.
\end{Proposition}

\begin{Proposition}\label{LieMJ}
Let $G$ be a reflexive topological abelian group, then   the mapping $\Phi_0:{\cal L}(G)\to {\rm CHom}_{co}(G^\wedge,\mathbb{R^\wedge})$ given by $\varphi\mapsto \varphi^{\wedge}$ is a topological isomorphism.
\end{Proposition}
\begin{proof}
Using Proposition \ref{LieLydia} we only need to check that $\Phi_0$ is onto:
Let $\psi: G^\wedge \to \mathbb{R^\wedge}$ be a continuous homomorphism and $\varphi=\alpha_G^{-1}\circ \psi^\wedge\circ \alpha_{\R}$. Let us see that $\varphi^{\wedge}=\psi$ or which is the same, that
for every $\chi\in G^\wedge$, $\varphi^{\wedge}(\chi)=\chi\circ \varphi=\chi\circ \alpha_G^{-1}\circ \psi^\wedge\circ \alpha_{\R}=\psi(\chi)$.

So, take $t\in \R$, if $x=(\alpha_G^{-1}\circ \psi^\wedge\circ \alpha_{\R})(t)=\alpha_G^{-1}( \psi)$, then $\alpha_G(x)=\alpha_{\mathbb R}(t)\circ \psi$

$\chi(x)=\alpha_G(x)(\chi)=(\alpha_{\mathbb R}(t)\circ \psi)(\chi)=\alpha_{\R}(t)(\psi(\chi))=\psi(\chi)(t)$,
Therefore $\varphi^{\wedge}(\chi)(t)=\psi(\chi)(t)$, for all $\chi \in G^\wedge$ and $t\in \R$, that is
$\varphi^{\wedge}=\psi$.
\end{proof}

\begin{Corollary}\label{LieR}
Let $G$ be a Hausdorff locally quasi--convex topological abelian group for which $\alpha_G$ is continuous, then ${\cal L}(G)$ is topologically isomorphic to a subgroup of ${\rm CHom}_{co}(G^\wedge,\mathbb{R})$. Moreover, if $G$ is
  reflexive  then    ${\cal L}(G) \cong {\rm CHom}_{co}(G^\wedge,\mathbb{R})$
\end{Corollary}

\begin{proof}
Take into account that $S: \R \to \R^\wedge$ given by $S(t)(s)=e^{2\pi ist}$ is a topological isomorphism.
\end{proof}

 We recall now that some  properties of the group preserved by ${\cal L}(G)$ are

 1. ${\cal L}(G)$ is Hausdorff, if the topological group $G$ is Hausdorff.

 2. ${\cal L}(G)$ is complete, if the topological group $G$ is complete.

 3. ${\cal L}(G)$ is a locally convex space,
  if the topological group $G$ is a Hausdorff locally quasi-convex group (see \cite[1.2]{LydiaLie}).

 4. ${\cal L}(G)$ is nuclear if the topological group $G$ is nuclear (see \cite[2.6]{LydiaLie}).

\begin{Proposition}
If $G$ is a locally quasi--convex metrizable and separable topological abelian group,  ${\rm CHom}_{co}(G^\wedge,\mathbb{R})$  is metrizable complete and separable. The Lie algebra ${\cal L}(G)$ is also metrizable and separable, and it is complete if the topological group $G$ is complete.
\end{Proposition}

\begin{proof}
Since $G$ is metrizable, $G^\wedge$ is hemicompact and a $k$-space \cite{MJ}. By the hemicompactness of $G^\wedge$, ${\rm CHom}_{co}(G^\wedge,\mathbb{R})$ is metrizable and because $G^\wedge$ is a $k$-space,  ${\rm CHom}_{co}(G^\wedge,\mathbb{R})$ is  complete.
On the other hand  since $G$ is metrizable and separable it holds (see \cite[1.7]{ChMPT}) that
   compact subsets of $G^\wedge$ are metrizable. But for a  hemicompact $k$-space $X$ whose compact subsets are metrizable, $C_{co}(X)$ is metrizable and separable (see \cite{KLMPT} and \cite{Warner}), hence  ${\rm CHom}_{co}(G^\wedge,\mathbb{R})$ is metrizable separable and complete. Observe now that ${\cal L}(G)$ is complete because $G$ is complete and it is metrizable and separable  because it is topologically isomorphic to a subgroup of  ${\rm CHom}_{co}(G^\wedge,\mathbb{R})$.
\end{proof}
{\bf Remark.}  For  torus groups
$\T^X$ were $X$ is an arbitrary set,we may identify the exponential function  with
the canonical quotient map $\R^X \to \T^X$, therefore, it is  open.  There are also compact
connected groups which are not torus groups but for which the exponential function
is open onto its image (see \cite{HMP}).  We next find other classes of abelian groups for which the
corestriction $\exp_{G}:{\cal L}(G)\longrightarrow G_a$ is open.

\begin{Theorem}\label{exp.open}
If $G$ is a  metrizable separable reflexive topological abelian group which has an arc--connected zero neighborhood, then the exponential mapping $\exp_{G}:{\cal L}(G)\longrightarrow G_a$ is a quotient map.
\end{Theorem}
\begin{proof}
  By the above Proposition, $ {\cal L}(G)$ is metrizable complete and separable. On the other hand, since  $G$ is a locally quasi-convex  topological group for which $\alpha_G$ is onto ${\rm im}\,{\rm exp}_{G}=G_a$. By hypothesis,  $G_a$ is an open  subgroup of the group $G$, hence it is closed. Therefore $G_a$ is metrizable complete. Since the exponential mapping is a continuous and onto homomorphism, by the open mapping theorem, it is open.
\end{proof}
\begin{Corollary}
Let  $G$ be a  metrizable separable reflexive  topological group:

1.  If $G$ is locally arc-connected,  then the exponential map $\exp_{G}:{\cal L}(G)\longrightarrow G_a$ is a quotient map.

2. If $G$ is  arc-connected,  then the exponential map $\exp_{G}:{\cal L}(G)\longrightarrow G$ is a quotient map.

3. If  the exponential map $\exp_{G}:{\cal L}(G)\longrightarrow G_a$ is a quotient map and $G$ has an arc--connected zero neighborhood, then $G$ is locally arc--connected.
\end{Corollary}
\begin{proof}
For 1) and 2) it is enough to know that metrizable reflexive groups are complete, so all the hypothesis of the above Corollary are fulfilled.

3)
Since ${\cal L}(G)$ is a locally convex topological
vector space it has arbitrarily small arc--connected neighborhoods of
zero which are mapped onto open identity neighborhoods of $G_a$ by $\exp_{G}$.
\end{proof}

The following  Lemma shows that for a topological abelian group, to have enough real characters  and to have enough characters that can be lifted, are equivalent properties.

\begin{Lemma}\label{propliftingofchars}\cite[2.1]{ATChD}
Let $G$ be a  topological abelian group and $e_G\neq x\in
G$. The following assertions are equivalent:

1. There exists a real character $f$ such that $f(x)\neq 0$.

2. There exists a character
$\varphi\in  G^\wedge_{\rm lift}$ such that $\varphi(x)\neq 1$.

\end{Lemma}
\begin{Corollary}
A Hausdorff topological abelian group $G$ has enough real characters if and only if $(G^{\wedge}_{\rm lift})^\bot=\{0\}$
\end{Corollary}

\begin{Theorem} \label{AChD}\cite{AChD}
Let $(G,\tau)$ be an abelian Hausdorff satisfying EAP.
 Then $$G^{\wedge}_{\rm lift} = {\rm im}\,{\rm exp}_{G^{\wedge}}= (G^{\wedge})_a$$
\end{Theorem}

\begin{Theorem}\label{section:maintheorem}
Let $(G,\tau)$ be an abelian Hausdorff group such that every arc in the character group is equicontinuous. Consider the
following conditions:
\begin{itemize}
\item[a)] $G$ has enough real characters.
\item[b)] $(G^{\wedge})_a$ is dense in $G^\wedge$, with the pointwise convergence topology.
\end{itemize}
Then a)$\Rightarrow$ b). If $G$ is a MAP group, b)$\Rightarrow$ a).
\end{Theorem}
\begin{proof}
By Lemma \ref{dense} and Theorem \ref{AChD}, $((G^{\wedge})_a)^\bot=(G^{\wedge}_{\rm lift})^\bot$. If the group $G$ has enough real characters $(G^{\wedge}_{\rm lift})^\bot$ is trivial and then $(G^{\wedge})_a$ is dense in $G^\wedge$, with the pointwise convergence topology. For the reverse implication take into account that for a MAP group $G$, $(G^{\wedge})_a$  dense in $G^\wedge$ with the pointwise convergence topology, implies $((G^{\wedge})_a)^\bot$ is trivial.
\end{proof}

\begin{Corollary}\label{LieMJ}
Let $G$ be a locally quasi-convex   topological abelian group such that $\alpha_G$ is onto then,
\begin{enumerate}
\item
$G_a$ is $\omega(G,G^\wedge)$--dense iff $G^\wedge$ has enough real characters.
\item
$G$ is arc--connected iff every character of $G^{\wedge}$ can be lifted.
\end{enumerate}
\end{Corollary}
\begin{proof}
(1) By Theorem \ref{main} and Lemma \ref{dense}, $G_a$ is $\omega(G,G^\wedge)$--dense iff $(G_a)^\bot=(\alpha_G^{-1}(G^{^{\wedge\wedge}}_{\rm lift}))^\bot=(G^{^{\wedge\wedge}}_{\rm lift})^\bot=\{0\}$, iff $G^\wedge$ has enough real characters.

(2) Again by Theorem \ref{main}, $G_a=G$ iff $G^{^{\wedge\wedge}}_{\rm lift}
=G^{^{\wedge\wedge}}$.
\end{proof}


\begin{Corollary}
Let $G$ be a locally quasi--convex topological abelian group with $\alpha_G$ onto and such that  closed subgroups are dually closed.
If $G^\wedge$ has enough real characters, then $G$ is connected.
\end{Corollary}
\begin{proof}
Having  $G^\wedge$
 enough real characters  the arc--component
$G_a$ of $G$ is  $\omega(G,G^\wedge)$--dense. Let $G_0$ be the
connected  component of $G$. It is clear that $G_a\subset G_0$
therefore $G_0$ is $\omega(G,G^\wedge)$--dense. Since $G_0$ is a closed
subgroup of the  group $G$,  $G_0$ is dually closed
and therefore it is $\omega(G,G^\wedge)$--closed. Then $G_0=G$, that
is: $G$ is connected.
\end{proof}
{\bf Examples}: Nuclear complete topological abelian groups are locally quasi--convex topological abelian groups with $\alpha_G$ onto and such that  closed subgroups are dually closed. Therefore every nuclear complete topological  group, such that its dual group, $G^\wedge$ has enough real characters, is connected.

The previous results allow us to give an alternative proof  for the following  well known fact.
\begin{Corollary}
Let $G$ be an LCA group then $G$ is connected iff $G^\wedge$ has enough real characters.
\end{Corollary}
\begin{proof}
Observe that LCA groups are nuclear complete, then the if part is true. Let us prove the reverse implication. Since $G$ is an LCA group, $G_a$ is dense in $G_0$. But $G_0=G$ because $G$ is connected, so $G_a$ is dense and therefore $\omega(G,G^\wedge)$ dense in $G$. Therefore, by Corollary \ref{LieMJ}, $G^\wedge$ has enough real characters.
\end{proof}

Acknowledgement. The author is indebted to Professors E. Martin Peinador and
X. Dominguez and S. Ardanza-Trevijano for very helpful suggestions.

\end{document}